\documentclass[smallextended]{svjour3} 

\def\makeheadbox{\relax}

\usepackage{amssymb}
\usepackage{epstopdf,comment}
\usepackage{  mathrsfs, enumerate}
\usepackage{amsmath}
\usepackage{enumitem}
\usepackage{hyperref}

\newcommand{\Var}{\operatorname{Var}}

\newcommand{\dom}{\operatorname{dom}}
\newcommand{\aff}{\operatorname{aff}}
\newcommand{\ri}{\operatorname{ri}}

\newcommand{\req}[1]{Eq.\,(\ref{#1})}

\begin{document}

\title{Distributionally Robust Variance Minimization: Tight Variance Bounds over $f$-Divergence Neighborhoods}
\titlerunning{ Distributionally Robust Variance Minimization}

\author{Jeremiah Birrell      }

\institute{Jeremiah Birrell \at
             TRIPODS Institute for Theoretical Foundations of Data Science,\\
University of Massachusetts Amherst\\
Amherst, MA, 01003, USA\\              
              \email{birrell@math.umass.edu}           
               }

\date{}

\maketitle

\begin{abstract}
Distributionally robust optimization (DRO) is a widely used framework for optimizing  objective functionals in the presence of both randomness and  model-form uncertainty. A key step in the practical solution of many DRO problems is  a tractable reformulation of the  optimization over the chosen model ambiguity set, which is generally  infinite dimensional.     Previous works have solved this problem in the case where the objective functional is an expected value. In this paper we study objective functionals that are the sum of an expected value and a variance penalty term.  We prove that the corresponding variance-penalized DRO  problem over an $f$-divergence neighborhood can be reformulated as a finite-dimensional convex optimization problem.  This result also provides tight uncertainty quantification bounds on the variance.


\end{abstract}

\section{Introduction}

Optimization problems  that depend on incompletely known parameter values or involve systems with inherently noisy dynamics  are often naturally phrased as stochastic programming (SP) problems of the general form
\begin{flalign}\label{eq:SP}
\text{{\bf (SP)}}&  \hspace{4.2cm} \min_{x\in \mathcal{X}} H[P,x]\,.&
\end{flalign}
The objective  functional, $H[P,x]$, depends on the underlying probability measure, $P$, which models the inherent randomness and/or parameter uncertainty, and on the control variable, $x$.  An important and much studied case of \eqref{eq:SP} is the minimization of an expected value, $H[P,x]=E_P[\rho_x]$,  for some $x$-dependent random variable, $\rho_x$, often thought of as a cost. In this paper we consider objective functionals of the form, $H[P,x]=E_P[\rho_x]+\Var_P[\phi_x]$, consisting of an expected cost and a variance penalty.  Such objective functionals arise in  resource allocation \cite{6483304},   stochastic control \cite{Li2003}, Markov decision processes \cite{10.2307/3689841},   and portfolio optimization  \cite{Yoshimoto,Tutuncu}, where the  variance penalty enforces a certain  risk aversion.

 In practice, the model $P$ is often learned from data.  Fitting a model  requires one to first make  various structural assumptions (e.g., Markovian versus non-Markovian or choosing a particular parametric family). This is a second source of uncertainty in the predictions of a model, beyond its  inherent probabilistic nature, which we term model-form uncertainty. This issue suggest that one generalize \eqref{eq:SP} to the following distributionally robust optimization (DRO) problem
\begin{flalign}\label{eq:DRO}
\text{{\bf (DRO)}}&  \hspace{3.5cm} \min_{x\in \mathcal{X}}\sup_{Q\in\mathcal{U}} H[Q,x]\,,&
\end{flalign}
i.e., minimizing the worst-case `cost' over the neighborhood of models (ambiguity set) $\mathcal{U}$. The ambiguity set  encodes the degree and form of uncertainty regarding the `true' model.   A key step in the practical solution of many DRO problems is a tractable reformulation of    the inner  optimization over $Q$, which is often infinite dimensional.  Finite dimensional reformulations are known when the objective functional is an expected value and for various types  of ambiguity sets, including moment constraints \cite{doi:10.1287/opre.1090.0795,doi:10.1287/opre.1090.0741,doi:10.1287/opre.2014.1314}, Kullback--Leibler  or $f$-divergence neighborhoods \cite{doi:10.1287/opre.1100.0821,Javid2012,Hu2013,10.2307/23359484,2016arXiv160509349L}, and Wasserstein neighborhoods \cite{2016arXiv160402199G,MohajerinEsfahani2018,doi:10.1287/moor.2018.0936}.  Such reformulations are also needed to solve problems with distributionally robust chance constraints  \cite{10.2307/2631503,Hu2013,Jiang2016,doi:10.1137/16M1094725}. An alternative approach to DRO over $f$-divergences neighborhoods  is the use of $f$-divergence penalties \cite{GOTOH2018448,2017arXiv171106565G}; in particular, \cite{GOTOH2018448} discusses the relation of that approach to variance penalties.

Considerably less is known about DRO for  more general objective functionals, beyond  expected values.  In this paper we study the following variance penalized DRO (VP-DRO) problem:
\begin{flalign}\label{eq:varDRO}
\text{{\bf (VP-DRO)}}&  \hspace{2cm} \min_{x\in \mathcal{X}}\sup_{Q\in\mathcal{U}} \{E_Q[\rho_x]+\Var_Q[\phi_x]\}\,.&
\end{flalign}
In    \cite{Tutuncu} VP-DRO was studied for moment constraints, resulting in an inherently  finite dimensional minimax problem; see also \cite{Scutella}. Here we consider $f$-divergence ambiguity sets, $\mathcal{U}_\eta(P)=\{Q:D_f(Q,P)\leq\eta\}$, where  $\eta>0$ and $P$ is a given baseline model.  In this case, the optimization over $Q$ in \eqref{eq:varDRO} is infinite dimensional and presents a considerable challenge on its own.  Our focus for the remainder of the paper will thus  be on the inner maximization problem in \eqref{eq:varDRO}. Specifically, we show it can be rewritten as the following finite dimensional convex optimization problem:
\begin{align}\label{eq:result_preview}
&\sup_{Q:D_f(Q,P)\leq\eta}\{E_Q[\rho]+\Var_Q[\phi]\}\\
=&\inf_{\lambda>0,\beta\in\mathbb{R},\nu\in\mathbb{R}}\left\{{\nu^2}/{4}+\beta+\eta\lambda +\lambda E_P[f^*((\rho+\phi^2-\nu\phi-\beta)/\lambda)]\right\}\,,\notag
\end{align}
where $f^*$ denotes the Legendre transform of $f$. Full details and assumptions can be found in Theorem \ref{thm:UQ_var_preview} below and the proof is given in Section \ref{sec:proof}.

In addition to its importance for DRO, note that \req{eq:result_preview} also  constitutes an uncertainty quantification (UQ) bound over the $f$-divergence model neighborhood, i.e., a bound on
\begin{flalign}\label{eq:UQ}
\text{{\bf (UQ)}}&  \hspace{2.5cm} \sup_{Q:D_f(Q,P)\leq\eta}\{H[Q,x]-H[P,x]\}\,,&
\end{flalign}
where $P$ is a known baseline model (typically one that is tractable, either numerically or analytically)  and $\eta>0$ is chosen so that  the `true' model (which is presumably too complex to work with directly) is contained in the model neighborhood $\mathcal{U}_\eta(P)=\{Q:D_f(Q,P)\leq\eta\}$. UQ bounds for expected values have been heavily studied  \cite{dupuis2011uq,dupuis2015path,BLM,BREUER20131552,glasserman2014,doi:10.1137/19M1237429,atar2015robust,DKPR}.  The formula \eqref{eq:result_preview}  extends  these works to provide a tight UQ bound on the variance.

\subsection{A Formal Argument}

The result \eqref{eq:result_preview} can be motivated by the following formal calculation: First recall that $\Var_Q[\phi]=\inf_{c\in\mathbb{R}}E_Q[(\phi-c)^2]$ and write
\begin{align}\label{eq:var_min}
&\sup_{Q:D_f(Q,P)\leq\eta}\{E_Q[\rho]+\Var_Q[\phi]\}=\sup_{Q:D_f(Q,P)\leq\eta}\inf_{c\in\mathbb{R}}\{E_Q[\rho+(\phi-c)^2]\}\,.
\end{align}
Assuming one can interchange the supremum and the infimum in \eqref{eq:var_min}, one can then use the tight bound on expected values over $f$-divergence neighborhoods from \cite{Javid2012} (see \req{eq:UQ_means} below) to write
\begin{align} 
&\sup_{Q:D_f(Q,P)\leq\eta}\{E_Q[\rho]+\Var_Q[\phi]\}\\
=&\inf_{c\in\mathbb{R},\lambda>0,\beta\in\mathbb{R}}\left\{(\beta+\eta\lambda )+\lambda E_P[f^*((\rho+(\phi-c)^2-\beta)/\lambda)]\right\}\,.\notag
\end{align}
By changing variables $\beta\to \beta+c^2$ and then $c=\nu/2$ one arrives at \req{eq:result_preview}.

The difficulty in making this formal calculation rigorous lies in the exchange of the supremum and the infimum in \eqref{eq:var_min}. Under rather strong assumptions this can be justified  by using, e.g., Sion's minimax theorem  \cite{sion1958,komiya1988} or Ky Fan's minimax theorem \cite{10.2307/88653}. However, the restrictions on $\rho$ and $\phi$ that would be required to satisfy the assumptions of these minmax results would significantly reduce the generality of our main theorem. For instance, both Sion's and  Ky Fan's theorems could be used if one  works on a Polish space and assumes, among other things, that $\rho+(\phi-c)^2$ is bounded above and upper semicontinuous for all $c$   (so as to ensure, via the Portmanteau theorem, that $Q\mapsto E_Q[\rho+(\phi-c)^2]$ is upper semicontinuous  in the Prohorov metric topology). The semicontinuity assumptions are often satisfied in practice but boundedness of the cost function is violated in a large number of applications, hence we do not wish to make these restrictive assumptions. We will take a different approach in the  proof of Theorem \ref{thm:UQ_var_preview}. Our result will follow from the solution of a more general convex optimization problem (see Section \ref{sec:proof}) and will  not rely on the technique from \req{eq:var_min}.

\subsection{Background on $f$-Divergences}\label{sec:background}
Before proceeding to the main theorem, we first provide some required background on $f$-divergences. For $-\infty\leq a<1<b\leq\infty$ we let  $\mathcal{F}_1(a,b)$ denote  the set of convex functions $f:(a,b)\to\mathbb{R}$ with $f(1)=0$. Such functions are continuous and extend to convex, lower semicontinuous functions $f:\mathbb{R}\to(-\infty,\infty]$ by defining $f(a)=\lim_{t\searrow a}f(t)$ and $f(b)=\lim_{t\nearrow b}f(t)$ (when either $a$ and/or $b$ is finite) and $f|_{[a,b]^c}=\infty$ (see Appendix \ref{app:convex_dis} for further details on $\mathcal{F}_1(a,b)$).  Functions $f\in \mathcal{F}_1(a,b)$ are appropriate for defining $f$-divergences as follows \cite{Broniatowski,Nguyen_Full_2010}: Let $\mathcal{P}(\Omega)$ denote the set of probability measures on  a measurable space  $(\Omega,\mathcal{M})$.  For $P,Q\in\mathcal{P}(\Omega)$ and $f\in \mathcal{F}_1(a,b)$ the $f$-divergence of $Q$  with respect to  $P$ is defined by
\begin{align}
D_f(Q,P)=\begin{cases} 
     E_P[f(dQ/dP)], & Q\ll P\\
      \infty, &Q\not\ll P
   \end{cases}\,.
\end{align}
We will also use the following variational characterization of $f$-divergences  \cite{Broniatowski,Nguyen_Full_2010}:
\begin{align}\label{eq:Df_variational}
D_f(Q,P)=&\sup_{\phi\in\mathcal{M}_b(\Omega)}\{E_Q[\phi]-E_P[f^*(\phi)]\}\,,
\end{align}
where $\mathcal{M}_b(\Omega)$ denotes the set of bounded measurable real-valued functions on $\Omega$ and $f^*$ is the Legendre transform of $f$.

\section{Tight Variance Bounds}

The main result  in this paper is the following tight  bound on  an expected value with  variance penalty over an $f$-divergence neighborhood:
\begin{theorem}\label{thm:UQ_var_preview}

Suppose:
\begin{enumerate}[label=\roman*.]
\item   $f\in\mathcal{F}_1(a,b)$ with $a\geq 0$.
\item  $P\in\mathcal{P}(\Omega)$. 
\item $\phi:\Omega\to\mathbb{R}$, $\phi \in L^1(P)$, and there exists $c_+,c_->0$, $\nu_+,\nu_-\in\mathbb{R}$ such that $E_P[[f^*( \pm c_\pm\phi-\nu_\pm)]^+]<\infty$.
\item $\rho:\Omega\to\mathbb{R}$ is measurable and if  $Q\in\mathcal{P}(\Omega)$ with $D_f(Q,P)<\infty$ then  $E_Q[\rho^-]<\infty$.
\end{enumerate}

Then $\phi\in L^1(Q)$ for all $Q\in\mathcal{P}(\Omega)$ that satisfy $D_f(Q,P)<\infty$ and for all  $\eta>0$ we have
\begin{align}\label{eq:UQ_var}
&\sup_{Q:D_f(Q,P)\leq\eta}\{E_Q[\rho]+\Var_Q[\phi]\}\\
=&\inf_{\lambda>0,\beta\in\mathbb{R},\nu\in\mathbb{R}}\left\{{\nu^2}/{4}+\beta+\eta\lambda +\lambda E_P[f^*((\rho+\phi^2-\nu\phi-\beta)/\lambda)]\right\}\,,\notag
\end{align}
where the map $(0,\infty)\times\mathbb{R}\times\mathbb{R}\to(-\infty,\infty]$, 
\begin{align}\label{eq:convex_objective}
(\lambda,\beta,\nu)\mapsto{\nu^2}/{4}+\beta+\eta\lambda +\lambda E_P[f^*((\rho+\phi^2-\nu\phi-\beta)/\lambda)]
\end{align}
is convex.

\end{theorem}
\begin{remark}
We use $g^\pm$ to denote the positive and negative parts of a (extended) real-valued function $g$, so that $g^\pm\geq 0$ and $g=g^+-g^-$.  The above assumptions imply $E_P[f^*((\rho+\phi^2-\nu\phi-\beta)/\lambda)]$ exists in $(-\infty,\infty]$ for all $\lambda>0$, $\beta\in\mathbb{R}$, $\nu\in\mathbb{R}$. Assumption (iv) is required to ensure that $E_Q[\rho]+E_Q[\phi^2]\neq -\infty+\infty$.  Often $\rho$ is a non-negative cost function and so this assumption is trivial. Otherwise,  Lemma \ref{lemma:EQ} below (applied to $-\rho$) provides a concrete condition that ensures condition (iv) holds. 
\end{remark}

Before proceeding to the proof of Theorem \ref{thm:UQ_var_preview} we  discuss several consequences of \req{eq:UQ_var}.  First we explore a pair of cases in which the    right hand side of \req{eq:UQ_var} can be simplified analytically. Then we examine the  formal solution of the minimization problem on the right hand side of \eqref{eq:UQ_var}.

\subsection{Tight Variance Bounds: Relative Entropy}
The relative entropy (i.e., Kullback--Leibler divergence) is the $f$ divergence corresponding to $f(t)=t\log(t)$, with Legendre transform $f^*(y)=e^{y-1}$ (we write $R(Q\|P)$ for $D_f(Q,P)$). The integrability assumption on $\phi$ (item iii. in Theorem \ref{thm:UQ_var_preview}) reduces to
\begin{align}
E_P[\exp(\pm c_{\pm}\phi)]<\infty
\end{align}
for some $c_+,c_->0$, i.e., $\phi$ must have a finite moment generating function under $P$ in a neighborhood of $0$. Assuming all of the conditions from Theorem \ref{thm:UQ_var_preview} hold, we can  evaluate the infimum over $\beta$ in \eqref{eq:UQ_var} to find
\begin{align}\label{eq:KL_objective}
&\sup_{Q:R(Q\|P)\leq\eta}\{E_Q[\rho]+\Var_Q[\phi]\}\\
=&\inf_{\lambda>0,\nu\in\mathbb{R}}\left\{{\nu^2}/{4}+\eta\lambda+\lambda \inf_{\beta\in\mathbb{R}}\left\{  \beta/\lambda+ e^{-\beta/\lambda-1}E_P[\exp((\rho+\phi^2-\nu\phi)/\lambda)]\right\}\right\}\notag\\
=&\inf_{\lambda>0,\nu\in\mathbb{R}}\left\{{\nu^2}/{4}+\eta\lambda+\lambda\log\left(E_P\left[e^{(\rho+\phi^2-\nu\phi)/\lambda}\right]\right)\right\}\,,\notag
\end{align}
where the optimum over $\beta$  occurs at $\beta_\lambda=\lambda(\log E_P[\exp((\rho+\phi^2-\nu\phi)/\lambda)]-1)$.

\subsection{Tight Variance Bounds: $\alpha$-Divergences}
The family of $\alpha$-divergences is defined via the convex function $f_\alpha(t)=\frac{t^\alpha-1}{\alpha(\alpha-1)}$, $t>0$, $\alpha>0$, $\alpha\neq 1$. For $\alpha>1$ the Legendre transform is
\begin{align}
f_\alpha^*(y)= y^{\alpha/(\alpha-1)}\alpha^{-1}(\alpha-1)^{\alpha/(\alpha-1)}1_{y>0}+\frac{1}{\alpha(\alpha-1)}\,.
\end{align}
The integrability assumption on $\phi$ (item iii. in Theorem \ref{thm:UQ_var_preview}) reduces to the requirement that
\begin{align}
E_P[ (\pm\phi-\nu_{\pm})^{\alpha/(\alpha-1)}1_{\pm \phi-\nu_{\pm}>0}]<\infty
\end{align}
for some $\nu_\pm\in\mathbb{R}$. In particular, this holds if $\phi$ has a finite $\alpha/(\alpha-1)$'th moment.

  For  $\alpha\in(0,1)$ the Legendre transform is
\begin{align}
f_\alpha^*(y)=\begin{cases} 
|y|^{-\alpha/(1-\alpha)}\alpha^{-1}(1-\alpha)^{-\alpha/(1-\alpha)}-\frac{1}{\alpha(1-\alpha)}\,,&y<0\\
\infty \, , &y\geq 0\,.
   \end{cases}
\end{align}
Here the $\alpha$-divergence has the upper bound $D_{f_\alpha}\leq \frac{1}{\alpha(1-\alpha)}$ and so one should assume $0<\eta<\frac{1}{\alpha(1-\alpha)}$. The integrability assumption on $\phi$ (item iii. in Theorem \ref{thm:UQ_var_preview}) reduces to the requirement that $\phi$  be bounded $P$-a.s. In this case, and assuming the conditions of Theorem \ref{thm:UQ_var_preview} hold, one can further evaluate the infimum over $\lambda$ in \eqref{eq:UQ_var} to find
\begin{align}\label{eq:alpha_div_simp}
&\sup_{Q:D_{f_\alpha}(Q,P)\leq\eta}\{E_Q[\rho]+\Var_Q[\phi]\}\\
=&\inf_{\beta\in\mathbb{R},\nu\in\mathbb{R}}\left\{{\nu^2}/{4}+\beta+\inf_{\lambda>0}\left\{\eta\lambda +\lambda\left( C_{\beta,\nu}\lambda^{\frac{\alpha}{1-\alpha}}-\frac{1}{\alpha(1-\alpha)}\right)\right\}\right\}\notag\\
=&\inf_{\beta\in\mathbb{R},\nu\in\mathbb{R}:C_{\beta,\nu}\neq\infty}\left\{{\nu^2}/{4}+\beta-\alpha\left(\frac{1-\alpha}{C_{\beta,\nu}}\right)^{\frac{1-\alpha}{\alpha}}\left(\frac{1}{\alpha(1-\alpha)}-\eta\right)^{\frac{1}{\alpha}}\right\}\,,\notag
\end{align}
where  $C_{\beta,\nu}\in(0,\infty]$ is given by
\begin{align}
C_{\beta,\nu}\equiv& \begin{cases} 
\frac{E_P[|\rho+\phi^2-\nu\phi-\beta|^{-\frac{\alpha}{1-\alpha}}]}{\alpha(1-\alpha)^{\frac{\alpha}{1-\alpha}}}\,,& P(\rho+\phi^2-\nu\phi-\beta\geq 0)=0\\
\infty\,, &\text{otherwise.}\notag
   \end{cases}
\end{align}

\subsection{Formal Solution of the Convex Minimization Problem}\label{sec:formal_sol}
After applying Theorem \ref{thm:UQ_var_preview}, the infinite-dimensional optimization problem over an $f$-divergence neighborhood is reduced to a finite-dimensional  convex minimization problem with objective function \eqref{eq:convex_objective} or, in special cases,  \eqref{eq:KL_objective} or \eqref{eq:alpha_div_simp}.  These objective functionals are expectations, and so one can apply standard  (stochastic) gradient descent methods to minimize them; we do not address convergence guarantees for such methods here, as that is outside the scope of the present work; see \cite{2018arXiv181008750D} for further discussion of this problem.  In this subsection we gain further  insight into these problems by formally solving the optimization  on the right hand side of \req{eq:UQ_var} (see   Section \ref{sec:proof} for the rigorous derivations): Fix $\eta>0$ and suppose the optimum in \eqref{eq:UQ_var} is achieved at $(\lambda_\eta,\beta_\eta,\nu_\eta)$.  Differentiating with respect to $\beta$  at the optimizer we find
\begin{align}
 E_P[(f^*)^\prime(\Psi_\eta)]=1\,,\,\,\,\,\,\Psi_\eta\equiv(\rho+\phi^2-\nu_\eta\phi-\beta_\eta)/\lambda_\eta\,.
\end{align}
$f^*$ is nondecreasing (see Appendix \ref{app:convex_dis}), hence  $(f^*)^\prime\geq 0$ and $dQ_\eta\equiv (f^*)^\prime(\Psi_\eta)dP$ is a probability measure. Next, by differentiating with respect to $\lambda$ we obtain
\begin{align}
\eta=&E_P[(f^*)^\prime(\Psi_\eta)\Psi_\eta-f^*(\Psi_\eta)]\,.
\end{align}
The function $g\equiv f^*$ is convex, hence $f(x)=g^*(x)=x (g^\prime)^{-1}(x)-g((g^\prime)^{-1}(x))$.  Letting $x=(f^*)^\prime(\Psi_\eta)$ we find
\begin{align}
\eta=&E_P[f((f^*)^\prime(\Psi_\eta)]=D_f(Q_\eta,P)\,.
\end{align}
Differentiating with respect to $\nu$ we obtain
\begin{align}
E_{Q_\eta}[\phi]=\nu_\eta/2.
\end{align}
Putting these together we can compute
\begin{align}
E_{Q_\eta}[\rho]+\Var_{Q_\eta}[\phi]=&\nu_\eta^2/4+\beta_\eta+\lambda_\eta E_{Q_\eta}[\Psi_\eta]\\
=&\nu_\eta^2/4+\beta_\eta+\lambda_\eta(\eta+ E_{P}[f^*(\Psi_\eta)])\,,\notag
\end{align}
which equals the right hand side of \eqref{eq:UQ_var}. Therefore the  probability measure that achieves the optimum on the left hand side of \eqref{eq:UQ_var} is the tilted measure
\begin{align}
dQ_\eta=(f^*)^\prime((\rho+\phi^2-\nu_\eta\phi-\beta_\eta)/\lambda_\eta)dP\,.
\end{align}

\section{Proof of Theorem \ref{thm:UQ_var_preview} }\label{sec:proof}
We now work towards  the proof of Theorem \ref {thm:UQ_var_preview}. We   will require a number of intermediate results, the first being a useful condition that ensures certain expectations exist.  
\begin{lemma}\label{lemma:EQ}
Let   $f\in\mathcal{F}_1(a,b)$ and  $P\in\mathcal{P}(\Omega)$. Suppose $\phi:\Omega\to\mathbb{R}$ is measurable and $E_P[[f ^*( c_0\phi-\nu_0)]^+]<\infty$ for some $\nu_0\in\mathbb{R}$ and $c_0>0$. Then for all $Q\in\mathcal{P}(\Omega)$ with  $D_f(Q,P)<\infty$ we have $E_Q[\phi^+]<\infty$.
\end{lemma}
\begin{proof}
Fix $b\in\mathbb{R}$ for which $f^*(b)$ is finite and define  
\begin{align}
\phi^+_n=\phi1_{0\leq \phi<n}+(b+\nu_0)/c_0 1_{\phi\not\in[0,n)}\,.
\end{align}
  Hence $c_0\phi_n^+-\nu_0 \in\mathcal{M}_b(\Omega)$ and the variational formula \eqref{eq:Df_variational} gives
\begin{align}
D_f(Q,P)\geq E_Q[c_0\phi_n^+-\nu_0]-E_P[f^*(c_0\phi_n^+-\nu_0)]\,,
\end{align}
where  $E_P[f^*(c_0\phi_n^+-\nu_0)]$ is defined in $(-\infty,\infty]$.  Hence
\begin{align}\label{eq:EQ_lemma1}
 E_Q[c_0\phi_n^+]-D_f(Q,P)\leq \nu_0+E_P[f^*(c_0\phi_n^+-\nu_0)].
\end{align}
We can bound
\begin{align}
f^*(c_0\phi_n^+-\nu_0)=&f^*(c_0\phi-\nu_0)1_{0\leq \phi<n}+f^*(b)1_{\phi\not\in[0,n)}\\
\leq& f^*(c_0\phi-\nu_0)^++|f^*(b)|\,,\notag
\end{align}
and so
\begin{align}
E_P[f^*(c_0\phi_n^+-\nu_0)]\leq E_P[f^*(c_0\phi-\nu_0)^+]+|f^*(b)|\,.
\end{align}
Combined with \req{eq:EQ_lemma1}, this implies
\begin{align}
 E_Q[c_0\phi_n^+]-D_f(Q,P)\leq \nu_0+ E_P[f^*(c_0\phi-\nu_0)^+]+|f^*(b)|<\infty
\end{align}
for all $n$. $\phi_n^+$ are uniformly bounded below, therefore  Fatou's Lemma implies
\begin{align}
E_Q[\liminf_n \phi_n^+]\leq& \liminf_n E_Q[\phi_n^+]\\
\leq& c_0^{-1}(\nu_0+ E_P[f^*(c_0\phi-\nu_0)^+]+|f^*(b)|+D_f(Q,P))\,.\notag
\end{align}
We have the pointwise limit $\phi_n^+\to \phi^++(b+\nu_0)/c_01_{\phi<0}$, hence
\begin{align}
&E_Q[\phi^+]+(b+\nu_0)/c_0Q(\phi<0)=E_Q[\liminf_n \phi_n^+]\\
\leq& c_0^{-1}(\nu_0+ E_P[f^*(c_0\phi-\nu_0)^+]+|f^*(b)|+D_f(Q,P))<\infty\,.\notag
\end{align}
This prove the claim.

\end{proof}

The following lemma is an intermediate step towards \req{eq:UQ_var}. It show how to express certain suprema over  $f$-divergence neighborhoods in terms of  finite dimensional maximin problems.
\begin{lemma}\label{lemma:gen_UQ_temp}
Suppose:
\begin{enumerate}[label=\roman*.]
\item   $f\in\mathcal{F}_1(a,b)$ with $a\geq 0$.
\item  $P\in\mathcal{P}(\Omega)$. 
\item $\phi_i:\Omega\to\mathbb{R}$, $i=1,...,k$, $\phi_i \in L^1(P)$, and there exists $c_i^+,c_i^->0$, $\nu_i^+,\nu_i^-\in\mathbb{R}$ such that  $E_P[[f^*( \pm c_i^\pm\phi_i-\nu_i^\pm)]^+]<\infty$ for $i=1,...,k$.
\item $\psi:\Omega\to\mathbb{R}$ is measurable and $\psi^-\in L^1(P)$.
\item $g:\mathbb{R}^k\to\mathbb{R}$ is convex.
\end{enumerate}

Then
\begin{enumerate}
\item $\phi_i\in L^1(Q)$, $i=1,...,k$,  for all $Q\in\mathcal{P}(\Omega)$  that satisfy $D_f(Q,P)<\infty$.
\item For all $\eta>0$ there exists $M^\pm_{i,\eta}\in\mathbb{R}$, $i=1,...,k$ such that 
\begin{align}\label{eq:EQ_phi_bound}
E_Q[\phi_i]\in[-M_{i,\eta}^-,M_{i,\eta}^+]
\end{align}
for all $i=1,...,k$ and all  $Q\in\mathcal{P}(\Omega)$  that satisfy $D_f(Q,P)\leq \eta$.
\item For all $\eta>0$ and all $C\subset \mathbb{R}^k$ with $\prod_{i=1}^k[-M_{i,\eta}^-,M_{i,\eta}^+]\subset C$  we have
\begin{align}\label{eq:gen_UQ_temp}
&\sup_{Q:D_f(Q,P)\leq\eta}\{E_Q[\psi]-g(E_Q[\phi_1],...,E_Q[\phi_k])\}\\
=&\sup_{z\in C}\inf_{\nu\in\mathbb{R}^k}\{\nu\cdot z-g(z)+\!\!\inf_{\lambda>0,\beta\in\mathbb{R}}\!\left\{\beta+\eta\lambda +\lambda E_P[f^*((\psi-\nu\cdot\phi-\beta)/\lambda)]\right\}\}\,,\notag
\end{align}
and  $E_P[f^*((\psi-\nu\cdot\phi-\beta)/\lambda)]$ exists in $(-\infty,\infty]$ for all $\lambda>0$, $\beta\in\mathbb{R}$, $\nu\in\mathbb{R}^k$.
\end{enumerate}
\end{lemma}
\begin{remark}
In \req{eq:gen_UQ_temp}, and in the following, we define $\infty-\infty\equiv\infty$ so that expectations are defined for all measurable functions; such a term is only possible in  \req{eq:gen_UQ_temp} on the left hand side and, under appropriate assumptions, can be ruled out entirely via Lemma \ref{lemma:EQ}.
\end{remark}
\begin{proof}
\begin{enumerate}
\item This is a direct consequence of Lemma \ref{lemma:EQ}.
\item 
Let $\eta>0$.  The tight bound on expected values over $f$-divergence neighborhoods from  \cite{Javid2012,duchi2018learning}, which we recall in Theorem \ref{thm:mean_UQ} in Appendix \ref{thm:mean_UQ}, implies
\begin{align}
E_Q[\pm\phi_i]\leq \frac{\eta}{c_i^\pm}+\frac{1}{c_i^\pm}(\nu_i^\pm+E_P[f^*(\pm c_i^\pm\phi_i-\nu_i^\pm)])\equiv M_{i,\eta}^\pm
\end{align}
for all  $Q\in\mathcal{P}(\Omega)$ with $D_f(Q,P)\leq\eta$. Assumption (iii) implies $M^\pm_{i,\eta}<\infty$.  The lower bound $f^*(y)\geq y$ (see Appendix \ref{app:convex_dis}) together with $\phi_i\in L^1(P)$ then implies $M^\pm_{i,\eta}\in\mathbb{R}$.

\item Here we will use Lemma \ref{lemma:constrained_disintegration} from Appendix \ref{app:convex_dis}  to rewrite the left hand side of \req{eq:gen_UQ_temp} as an optimization of  expected values over an $f$-divergence neighborhood and then employ Theorem \ref{thm:mean_UQ} from Appendix \ref{app:UQ_means}.  We start by fixing $\eta>0$, letting $M_{i,\eta}^\pm$ be as in part (2) and letting  $C\subset \mathbb{R}^k$ with $\prod_{i=1}^k[-M_{i,\eta}^-,M_{i,\eta}^+]\subset C$.   Using $f^*(y)\geq y$, for  $\lambda>0$, $\beta\in\mathbb{R}$, $\nu\in\mathbb{R}^k$ we find
\begin{align}
f^*((\psi-\nu\cdot\phi-\beta)/\lambda)^-\leq\frac{1}{\lambda}(\psi^-+(\nu\cdot\phi+\beta)^+)\in L^1(P).
\end{align}
Hence $E_P[f^*((\psi-\nu\cdot\phi-\beta)/\lambda)]$ exists in $(-\infty,\infty]$. The following formula will be key during the remainder of the derivation: For any $\nu\in\mathbb{R}^k$ we have $(\psi-\nu\cdot\phi)^-\in L^1(P)$ and so Theorem \ref{thm:mean_UQ} implies
\begin{align}\label{eq:EQ_nu}
&\sup_{Q:D_f(Q,P)\leq\eta}E_Q[\psi-\nu\cdot\phi]\\
=&\inf_{\lambda>0,\beta\in\mathbb{R}}\{\beta+\eta\lambda+\lambda E_P[f^*((\psi-\nu\cdot\phi-\beta)/\lambda)]\}\,,\notag
\end{align}
where $\infty-\infty\equiv\infty$.

 To prove the claimed equality \eqref{eq:gen_UQ_temp} we consider two cases. First, suppose there exists $\nu_0\in\mathbb{R}^k$ such that $E_P[f^*((\psi-\nu_0\cdot\phi-\beta)/\lambda)^+]=\infty$ for all $\beta\in\mathbb{R}$, $\lambda>0$:  Taking  $\nu=\nu_0$ in \req{eq:EQ_nu} we see that
\begin{align}
\sup_{Q:D_f(Q,P)\leq\eta}E_Q[\psi-\nu_0\cdot\phi]= \infty\,.
\end{align}
For any other $\nu$ we have 
\begin{align}
E_Q[(\nu-\nu_0)\cdot\phi]\geq -\sum_i|(\nu-\nu_0)^i|\max\{M_{i,\eta}^-,M_{i,\eta}^+\}\equiv -C
\end{align}
and hence
\begin{align}\label{eq:case1_2}
\sup_{Q:D_f(Q,P)\leq\eta}E_Q[\psi-\nu\cdot\phi]=\infty
\end{align}
for all $\nu\in\mathbb{R}^k$. \req{eq:EQ_nu} then implies
\begin{align}
\inf_{\lambda>0,\beta\in\mathbb{R}}\{\beta+\eta\lambda+\lambda E_P[f^*((\psi-\nu\cdot\phi-\beta)/\lambda)]\}=\infty
\end{align}
for all $\nu\in\mathbb{R}^k$ and so we see that the right hand side of \req{eq:gen_UQ_temp} equals $+\infty$.

\req{eq:EQ_phi_bound} together with continuity of $g$ imply that there exists $D\in\mathbb{R}$ such that $g(E_Q[\phi])\leq D$ for all $Q\in\mathcal{P}(\Omega)$ with $D_f(Q,P)\leq\eta$. Therefore
\begin{align}
&\sup_{Q:D_f(Q,P)\leq\eta}\{E_Q[\psi]-g(E_Q[\phi])\}\geq -D+\sup_{Q:D_f(Q,P)\leq\eta}E_Q[\psi]\,.
\end{align}
\req{eq:case1_2} for $\nu=0$ then implies that the left hand side of  \req{eq:gen_UQ_temp} also equals $+\infty$.

Now consider the alternative case, where for all $\nu\in\mathbb{R}^k$ there exists $\beta\in\mathbb{R}$, $\lambda>0$ such that $E_P[f^*((\psi-\nu\cdot\phi-\beta)/\lambda)^+]<\infty$.  We will show that the optimization problem $\inf_{Q\in {X}} F(Q,H(Q))$ can be written as an iterated optimization over level sets of $H$.  To do this we will use the convex optimization result given in Lemma \ref{lemma:constrained_disintegration} of Appendix \ref{app:convex_dis}.  This result relies  on a variant of the Slater conditions, which we now verify:  Let $V$ be the  vector space consisting of all finite, real linear combinations of measures in $\{Q\in\mathcal{P}(\Omega): D_f(Q,P)<\infty\}$ and let
\begin{align}
X=\{Q:D_f(Q,P)\leq\eta,\, E_Q[\psi^-]<\infty\}.
\end{align}
Then $P\in X\subset V$, $X$ is convex, and $\psi\in L^1(Q)$ for all $Q\in X$ (Theorem \ref{thm:mean_UQ} implies that $E_Q[\psi^+]<\infty$ for all $Q$ with $D_f(Q,P)\leq\eta$).  Define the convex function $F:X\times\mathbb{R}^k\to\mathbb{R}$ by $F(Q,z)=g(z)-E_Q[\psi]$ and the linear function $H:V\to\mathbb{R}^k$, $H(\mu)=(\int\phi_1d\mu,...,\int \phi_k d\mu)$.  For all $\nu\in\mathbb{R}^k$, $z\in\mathbb{R}^k$ we can use \req{eq:EQ_nu} to compute
\begin{align}\label{eq:Slater_functional}
&\inf_{Q\in X}\{F(Q,z)+\nu\cdot(H(Q)-z)\}\\
=&g(z)-\nu\cdot z-\sup_{Q:D_f(Q,P)\leq\eta}\{E_Q[\psi-\nu\cdot\phi]\}\notag\\
=&g(z)-\nu\cdot z-\inf_{\lambda>0,\beta\in\mathbb{R}}\{\beta+\eta\lambda+\lambda E_P[f^*((\psi-\nu\cdot\phi-\beta)/\lambda)]\}\,.\notag
\end{align}
In the case currently under consideration, there exists $\lambda>0$, $\beta$ such that  $E_P[f^*((\psi-\nu\cdot\phi-\beta)/\lambda)^+]<\infty$ and so $\inf_{Q\in X}\{F(Q,z)+\nu\cdot(H(Q)-z)\}>-\infty$. 

With this we have  shown that all of the hypotheses of Lemma \ref{lemma:constrained_disintegration} from  Appendix \ref{app:convex_dis} hold and hence  we obtain the following: For all $K\subset \mathbb{R}^k$ with $H(X)\subset K$ we have
 \begin{align}
&\inf_{Q\in {X}} F(Q,H(Q))=\inf_{z\in K}\sup_{\nu\in\mathbb{R}^k}\inf_{Q\in X}\{F(Q,z)+\nu \cdot(H(Q)-z)\}\,,
\end{align}
i.e.,   
 \begin{align}
&\sup_{Q:D_f(Q,P)\leq\eta}\{E_Q[\psi]-g(E_Q[\phi_1],...,E_Q[ \phi_k])\}\\
=&\sup_{z\in K}\inf_{\nu\in\mathbb{R}^k}\{\nu\cdot z-g(z)+\!\!\inf_{\lambda>0,\beta\in\mathbb{R}}\!\{\beta+\eta\lambda+\lambda E_P[f^*((\psi-\nu\cdot\phi-\beta)/\lambda)])\}\}\,.\notag
\end{align}
From part (2) we see that $H(X)\subset\prod_{i=1}^k[-M_{i,\eta}^-,M_{i,\eta}^+]\subset C$ and so we can take $K=C$, thus completing  the proof.

\end{enumerate}
\end{proof}

Given further assumptions on $g$ we can exchange the order of the minimization and maximization in  \req{eq:gen_UQ_temp} and  evaluate the supremum over $z$, thereby expressing the result as a finite dimensional convex minimization problem.
\begin{theorem}\label{thm:gen_UQ}
Suppose:
\begin{enumerate}[label=\roman*.]
\item   $f\in\mathcal{F}_1(a,b)$ with $a\geq 0$.
\item  $P\in\mathcal{P}(\Omega)$.
\item $\phi_i:\Omega\to\mathbb{R}$, $i=1,...,k$, $\phi_i \in L^1(P)$, and there exists $c_i^+,c_i^->0$, $\nu_i^+,\nu_i^-\in\mathbb{R}$ such that  $E_P[[f^*( \pm c_i^\pm\phi_i-\nu_i^\pm)]^+]<\infty$ for $i=1,...,k$.
\item $\psi:\Omega\to\mathbb{R}$ is measurable and $\psi^-\in L^1(P)$.
\item $g:\mathbb{R}^k\to\mathbb{R}$ is convex, $C^1$,  and
\begin{align}
\lim_{z\to\infty}g(z)/\|z\|=\lim_{z\to\infty}\|\nabla g(z)\|=\lim_{\nu\to\infty}g^*(\nu)/\|\nu\|=\infty\,.
\end{align}
\end{enumerate}

Let   $\eta>0$. Then
\begin{enumerate}
\item 
\begin{align}\label{eq:gen_UQ_final}
&\sup_{Q:D_f(Q,P)\leq\eta}\{E_Q[\psi]-g(E_Q[\phi_1],...,E_Q[\phi_k])\}\\
=&\inf_{\lambda>0,\beta\in\mathbb{R},\nu\in\mathbb{R}^k}\left\{g^*(\nu)+\beta+\eta\lambda +\lambda E_P[f^*((\psi-\nu\cdot\phi-\beta)/\lambda)]\right\}\,.\notag
\end{align}
\item The map $(0,\infty)\times\mathbb{R}\times\mathbb{R}^k\to(-\infty,\infty]$, 
\begin{align}\label{eq:obj_fun_gen}
(\lambda,\beta,\nu)\mapsto g^*(\nu)+\beta+\eta\lambda +\lambda E_P[f^*((\psi-\nu\cdot\phi-\beta)/\lambda)]
\end{align}
is convex.

\end{enumerate}
\end{theorem}
\begin{proof}
First we collect some useful facts regarding $g^*$: We have assumed  $g$ is $C^1$ and $\lim_{z\to\infty}g(z)/\|z\|=\infty$, therefore for all $\nu\in\mathbb{R}^k$ we have $g^*(\nu)=\nu\cdot z^*_\nu-g(z^*_\nu)$ for some $z^*_\nu\in\mathbb{R}^k$ with $\nabla g(z^*_\nu)=\nu$. In particular, $g^*$ is a real valued convex function on $\mathbb{R}^k$ and hence is   continuous.

 By convexity of $f^*$ and of the perspective of a convex function we see that the map $(0,\infty)\times\mathbb{R}\times\mathbb{R}^k\to(-\infty,\infty]$, 
\begin{align}
(\lambda,\beta,\nu)\mapsto \lambda E_P[f^*((\psi-\nu\cdot \phi-\beta)/\lambda)]
\end{align}
is convex.   Together with the convexity of $g^*$, we therefore conclude that  \eqref{eq:obj_fun_gen} is convex.

To prove (1), first use Lemma \ref{lemma:gen_UQ_temp} to conclude the following: For all $Q\in\mathcal{P}(\Omega)$  with $D_f(Q,P)<\infty$ we have $\phi_i\in L^1(Q)$, $i=1,...,k$,  and  for all $C\subset \mathbb{R}^k$ with $\prod_{i=1}^k[-M_{i,\eta}^-,M_{i,\eta}^+]\subset C$  we have
\begin{align}
&\sup_{Q:D_f(Q,P)\leq\eta}\{E_Q[\psi]-g(E_Q[\phi_1],...,E_Q[\phi_k])\}\\
=&\sup_{z\in C}\inf_{\nu\in\mathbb{R}^k}\{\nu\cdot z-g(z)+\inf_{\lambda>0,\beta\in\mathbb{R}}\left\{\beta+\eta\lambda+\lambda E_P[f^*((\psi-\nu\cdot\phi-\beta)/\lambda)]\right\}\}\,.\notag
\end{align}
The claimed equality, \req{eq:gen_UQ_final}, will follow if we can show that
\begin{align}\label{eq:UQ_gen_goal}
&\sup_{z\in C}\inf_{\nu\in\mathbb{R}^k}\{\nu\cdot z-g(z)+\inf_{\lambda>0,\beta\in\mathbb{R}}\left\{\beta+\eta\lambda +\lambda E_P[f^*((\psi-\nu\cdot\phi-\beta)/\lambda)]\right\}\}\\
=&\inf_{\lambda>0,\beta\in\mathbb{R},\nu\in\mathbb{R}^k}\left\{g^*(\nu)+\beta+\eta\lambda +\lambda E_P[f^*((\psi-\nu\cdot\phi-\beta)/\lambda)]\right\}\notag
\end{align}
for some set $C$ containing $\prod_{i=1}^k[-M_{i,\eta}^-,M_{i,\eta}^+]$.

\req{eq:UQ_gen_goal} is trivial if $E_P[f^*((\psi-\nu\cdot\phi-\beta)/\lambda)]=\infty$ for all $\nu\in\mathbb{R}^k$, $\lambda>0$, $\beta\in\mathbb{R}$, so suppose not.  First we rewrite the left hand side of \req{eq:UQ_gen_goal} in a more convenient form: The map  
\begin{align}
(\lambda,\beta,\nu)\mapsto \beta+\eta\lambda+\lambda E_P[f^*((\psi-\nu\cdot \phi-\beta)/\lambda)]
\end{align}
is  convex in $(\lambda,\beta,\nu)$ on $(0,\infty)\times\mathbb{R}\times\mathbb{R}^k$. Hence, minimizing over  $\lambda$ and $\beta$ results in a convex function of $\nu$, i.e.,
\begin{align}\label{eq:h_def}
h:\nu\in\mathbb{R}^k\mapsto \inf_{\lambda>0,\beta\in\mathbb{R}}\{ \beta+\eta\lambda+\lambda E_P[f^*((\psi-\nu\cdot \phi-\beta)/\lambda)]\}
\end{align}
is also convex, provided  that $h>-\infty$ (see Proposition 2.22 in \cite{rockafellar2009variational}). To see that $h>-\infty$, use the fact that  $f^*(y)\geq y$ to compute
\begin{align}
\inf_{\lambda>0,\beta\in\mathbb{R}}\{ \beta+\eta\lambda+\lambda E_P[f^*((\psi-\nu\cdot \phi-\beta)/\lambda)]\}\geq E_P[\psi]-E_P[\nu\cdot\phi]>-\infty.
\end{align}
Therefore Lemma   \ref{lemma:convex_min_ri} in Appendix \ref{app:convex_dis} implies that the infimum can be restricted to the relative interior of the domain of $h$:
\begin{align}\label{eq:gen_UQ_temp2}
&\sup_{z\in C}\inf_{\nu\in\mathbb{R}^k}\{\nu\cdot z-g(z)+\inf_{\lambda>0,\beta\in\mathbb{R}}\left\{\beta+\eta\lambda +\lambda E_P[f^*((\psi-\nu\cdot\phi-\beta)/\lambda)]\right\}\}\\
=&\sup_{z\in C}\inf_{\nu\in \ri(\dom h)}\{\nu\cdot z-g(z)+h(\nu)\}\,.\notag
\end{align}

 We now proceed to  show that the required equality \eqref{eq:UQ_gen_goal} holds for  $C=[-R,R]^k$ when $R$ is sufficiently large. We restrict to $R\geq R_0\equiv \max\{M_{i,\eta}^\pm\}$ so that $\prod_i[-M_{i,\eta}^-,M_{i,\eta}^+]\subset [-R,R]^k$.  We will now show that the supremum over $z$ and infimum over $\nu$ in \req{eq:gen_UQ_temp2} can be commuted: The map  $(\nu,z)\to \nu\cdot z-g(z)+h(\nu)$ is continuous on $\ri(\dom h)\times[-R,R]^k$, concave in $z$, convex in $\nu$.  The domain of  $z$ is compact, hence Sion's  minimax theorem (see \cite{sion1958,komiya1988}) implies
\begin{align}\label{eq:gen_UQ_goal2}
&\sup_{z\in[-R,R]^k}\inf_{\nu\in\ri(\dom h)}\{\nu\cdot z-g(z)+h(\nu)\}\\
=&\inf_{\nu\in \ri(\dom h)}\{\sup_{z\in[-R,R]^k}\{\nu\cdot z-g(z)\}+h(\nu)\}\leq\inf_{\nu\in \ri(\dom h)}\{g^*(\nu)+h(\nu)\}\,.\notag
\end{align}
Next we  show that equality holds in the last line of \eqref{eq:gen_UQ_goal2} when $R$ is sufficiently large:  We are in the case where  $E_P[f^*((\psi-\nu\cdot\phi-\beta)/\lambda)]<\infty$ for some $\nu_0\in\mathbb{R}^k$, $\lambda_0>0$, $\beta_0\in\mathbb{R}$ and so 
\begin{align}\label{eq:bounded_above}
\inf_{\nu\in \ri(\dom h)}\{g^*(\nu)+h(\nu)\}\leq g^*(\nu_0)+h(\nu_0)<\infty\,.
\end{align}
Convexity of $h$ implies that it has an affine lower bound. Hence there exists $D\geq 0$ such that $h(\nu)\geq -D\|\nu\|_1+d$ for all $\nu$.  Fix $\widetilde R>\max\{R_0,D\}$ and choose $\widetilde C>0$ such that
\begin{align}
&\widetilde R(1-D/\widetilde R)\widetilde C+d\\
> &\inf_{\nu\in\ri(\dom h)}\{g^*(\nu)+h(\nu)\}+\max_{w:w_i\in\{\pm \widetilde R\}, i=1,...,k} g(w)\,.\notag
\end{align}
Note that this is possible because of \req{eq:bounded_above}.  Finally,  $\lim_{z\to\infty}\|\nabla g(z)\|=\infty$ and so we can choose $R>\widetilde R$ such that   $\|\nabla g(z)\|_1>\widetilde C$ for all $z\not\in[-R,R]^k$.

To prove equality in \req{eq:gen_UQ_goal2}, let $\nu\in \ri(\dom h)$ and consider the following two cases.
\begin{enumerate}[label=\alph*.]
\item Suppose $\{z:\nabla g(z)=\nu\}\subset [-R,R]^k$:  We know that $g^*(\nu)=\nu\cdot z_\nu^*-g(z_\nu^*)$ where $\nabla g(z_\nu^*)=\nu$.  Therefore  $z_\nu^*\in[-R,R]^k$ and so
\begin{align}
&h(\nu)+\sup_{z\in[-R,R]^k}\{\nu\cdot z-g(z)\}\geq h(\nu)+ \nu\cdot z_\nu^*-g(z_\nu^*)\\
\geq&  h(\nu)+g^*(\nu)\geq\inf_{\nu\in\ri(\dom h)}\{g^*(\nu)+h(\nu)\}\,.\notag
\end{align}

\item Suppose there exists $z_0\not\in  [-R,R]^k$ with $\nabla g(z_0)=\nu$:  Let $w_i=\text{sgn}(\nu_i)\widetilde R$ so  that $w\in[-R,R]^k$, $\nu\cdot w=\widetilde R\|\nu\|_1$ and
\begin{align}
&h(\nu)+\sup_{z\in[-R,R]^k}\{\nu\cdot z-g(z)\}\geq h(\nu)+\widetilde R\|\nu\|_1-g(w)\\
\geq&\widetilde R(1-D/\widetilde R)\|\nabla g(z_0)\|_1+d-g(w)\,.\notag
\end{align}
The definitions of $R$ and $\widetilde R$ imply  $\|\nabla g(z_0)\|_1>\widetilde C$ and $\widetilde R(1-D/\widetilde R)>0$, hence
\begin{align}
&h(\nu)+\sup_{z\in[-R,R]^k}\{\nu\cdot z-g(z)\}\geq  \widetilde R(1-D/\widetilde R)\widetilde C+d-g(w)\\
\geq& \inf_{\nu\in\ri(\dom h)}\{g^*(\nu)+h(\nu)\}+\max_{w:w_i\in\{\pm \widetilde R\}, i=1,...,k} \{g(w)\}-g(w)\notag\\
\geq&\inf_{\nu\in\ri(\dom h)}\{g^*(\nu)+h(\nu)\}\,.\notag
\end{align}

\end{enumerate}
Combining these two cases, we see that
\begin{align}
&\inf_{\nu\in\ri(\dom h)}\left\{\sup_{z\in[-R,R]^k}\{\nu\cdot z-g(z)\}+h(\nu)\right\}\geq  \inf_{\nu\in\ri(\dom h)}\{g^*(\nu)+h(\nu)\}\,.
\end{align}
Therefore we have equality in  \req{eq:gen_UQ_goal2}. Applying Lemma \ref{lemma:convex_min_ri} to the continuous function $g^*$ and the convex function $h$ and then using the formula \eqref{eq:h_def} for $h$ we arrive at
\begin{align}\label{eq:gen_UQ_temp3}
&\sup_{z\in[-R,R]^k}\inf_{\nu\in\ri(\dom h)}\{\nu\cdot z-g(z)+h(\nu)\}=\inf_{\nu\in \mathbb{R}^k}\{g^*(\nu)+h(\nu)\}\\
=&\inf_{\lambda>0,\beta\in\mathbb{R},\nu\in\mathbb{R}^k}\{ g^*(\nu)+\beta+\eta\lambda+\lambda E_P[f^*((\psi-\nu\cdot \phi-\beta)/\lambda)]\}\,.\notag
\end{align}
Combining   \eqref{eq:gen_UQ_temp2} with \eqref{eq:gen_UQ_temp3} we arrive at \eqref{eq:UQ_gen_goal} and so the result is proven.

\end{proof}
\begin{corollary}
Applying Theorem \ref{thm:gen_UQ}  to $\phi$, $\psi\equiv \rho+ \phi^2$,  $g(z)=z^2$, and $g^*(\nu)=\nu^2/4$ we obtain the tight variance bound stated in Theorem \ref{thm:UQ_var_preview}.
\end{corollary}

\appendix

\section{ Tight  Bounds on Expected Values}\label{app:UQ_means}

A key ingredient in the proof of Theorem \ref{thm:UQ_var_preview} is the following tight  bound on expected values   over $f$-divergence neighborhoods which was proven for bounded integrands, $\phi$, in Theorem 1 of \cite{Javid2012} and was extended to  unbounded integrands in \cite{duchi2018learning}. 
\begin{theorem}\label{thm:mean_UQ}
Suppose:
\begin{enumerate}[label=\roman*.]
\item   $f\in\mathcal{F}_1(a,b)$ with $a\geq 0$.
\item  $P\in\mathcal{P}(\Omega)$. 
\item $\phi:\Omega\to\mathbb{R}$ is measurable and $\phi^- \in L^1(P)$.
\end{enumerate}

Then for all $\eta>0$ we have
\begin{align}\label{eq:UQ_means}
\sup_{Q:D_f(Q,P)\leq \eta}E_Q[\phi]=&\inf_{\lambda>0,\beta\in\mathbb{R}}\left\{(\beta+\eta\lambda )+\lambda E_P[f^*((\phi-\beta)/\lambda)]\right\}\,,
\end{align}
where we use the convention $\infty-\infty\equiv\infty$ to extend the definition of $E_Q[\phi]$  to all $Q$.
\end{theorem}
\begin{remark}
 Note that   $f^*(y)\geq y$ (see Appendix \ref{app:convex_dis}) implies
\begin{align}
f^*((\phi-\beta)/\lambda)\geq (\phi-\beta)/\lambda\in L^1(P)
\end{align}
for all $\beta\in\mathbb{R}$, $\lambda>0$, and hence $E_P[f^*((\phi-\beta)/\lambda)]$ in \req{eq:UQ_means} exists in $(-\infty,\infty]$.
\end{remark}

\section{Convex Functions and Optimization}\label{app:convex_dis}
Here we recall several key results from convex analysis and optimization that are needed above, as will as fix our notation; for much more detail on these subjects see, e.g.,   \cite{rockafellar1970convex,luenberger1997optimization,rockafellar2009variational}. We will also prove a pair of  lemmas that are  key tools in the proof of Theorem \ref{thm:UQ_var_preview} above. 

 We will denote the interior of a set $A\subset\mathbb{R}^n$ by $A^o$. We denote the relative interior of a convex set $C\subset\mathbb{R}^n$ by $\ri(C)$, the affine hull by $\aff(C)$, and the domain of   a convex function $f:C\to(-\infty,\infty]$ by $\dom f\equiv\{f<\infty\}$ (we do not allow convex functions to take the value $-\infty$ in this work). Recall that a convex function on $\mathbb{R}^n$ is continuous on the relative interior of its domain (see Theorem 10.1 in \cite{rockafellar1970convex}).   The Legendre transform of $f$, defined by
\begin{align}
f^*(y)=\sup_{x\in\mathbb{R}^n}\{yx-f(x)\}\,,
\end{align}
is a convex lower semicontinuous function on $\mathbb{R}^n$, provided that $f$ is not identically $\infty$.    We will  be especially concerned with the following classes of convex functions, which are used to define $f$-divergences:  For $-\infty\leq a<1<b\leq\infty$ we let  $\mathcal{F}_1(a,b)$ be  the set of convex functions $f:(a,b)\to\mathbb{R}$ with $f(1)=0$. Such functions are continuous and extend to convex, lower semicontinuous  functions, $f:\mathbb{R}\to(-\infty,\infty]$, by defining $f(a)=\lim_{t\searrow a}f(t)$ and $f(b)=\lim_{t\nearrow b}f(t)$ (where appropriate; if  $a$ and/or $b$ is finite then the corresponding limit is guaranteed to exist in $(-\infty,\infty]$) and $f|_{[a,b]^c}=\infty$. The Legendre transform of this extension can be computed via
\begin{align}\label{eq:f_star}
f^*(y)=\sup_{x\in (a,b)}\{yx-f(x)\}\,.
\end{align}
Note that  $f(1)=0$ implies that $f^*(y)\geq y$ for all $y\in\mathbb{R}$.  From \req{eq:f_star} we also see that if $a\geq 0$ then $f^*$ is nondecreasing and hence $(-\infty,d]\subset\dom f^*$ for some $d\in\mathbb{R}$.

The following lemma shows that to minimize the sum of a continuous and convex function, it suffices to minimize over the relative interior of the domain.
\begin{lemma}\label{lemma:convex_min_ri}
Let $g:\mathbb{R}^n\to\mathbb{R}$ be continuous and $f:\mathbb{R}^n\to(-\infty,\infty]$ be convex.  Then
\begin{align}
\inf_{\mathbb{R}^n}\{g+f\}=\inf_{\ri(\dom f)}\{g+f\}.
\end{align}
\end{lemma}
\begin{proof}
The result is trivial if $\dom f=\emptyset$ so suppose not. Then $\ri(\dom f)\neq \emptyset$ and we can fix $x_0\in \ri(\dom f)$. For $x\in \dom f$ we have $(1-t)x+tx_0\in \ri(\dom f)$ for all $t\in(0,1)$ (see Theorem 6.1 in \cite{rockafellar1970convex}) therefore, using convexity of $f$, we find
\begin{align}
\inf_{\ri(\dom f)}\{g+f\}\leq& g((1-t)x+tx_0)+ f((1-t)x+tx_0)\\
\leq& g((1-t)x+tx_0)+ (1-t)f(x)+tf(x_0)\notag
\end{align}
for all $t\in(0,1)$. By taking  $t\searrow 0$ we arrive at
\begin{align}\label{eq:inf_ri_ineq}
\inf_{\ri(\dom f)}\{g+f\}\leq g(x)+ f(x)
\end{align}
for all $x\in \dom f$. \req{eq:inf_ri_ineq} trivially holds when $x\not\in \dom f$ and hence we obtain
\begin{align}
\inf_{\ri(\dom f)}\{g+f\}\leq \inf_{\mathbb{R}^n}\{ g+ f\}\,.
\end{align}
The reverse inequality is trivial.

\end{proof}

The following lemma is one of our  key technical tools. It splits a convex minimization problem into an iterated optimization of Lagrangians over the level sets of a given linear function and relies on a variant of the Slater conditions. This result is new, to the best of the author's knowledge.
\begin{lemma}\label{lemma:constrained_disintegration} 
Let $V$ be a real vector space  and suppose:
\begin{enumerate}[label=\roman*.]
\item $X\subset V$, $C\subset\mathbb{R}^k$ are nonempty convex subsets,
\item $f:X\times C\to \mathbb{R}$ is jointly convex,
\item $h:V\to\mathbb{R}^k$ is linear with $h(X)\subset C$,
\item $\inf_{x\in X}\{f(x,z)+\nu \cdot(h(x)-z)\}>-\infty$ for all $\nu\in\mathbb{R}^k$, $z\in C$.
\end{enumerate}
Then for all $K\subset \mathbb{R}^k$ with $h(X)\subset K\subset C$ we have
 \begin{align}\label{eq:constrained_disint}
\inf_{x\in {X}} f(x,h(x))=\inf_{z\in K}\sup_{\nu\in\mathbb{R}^k}\inf_{x\in X}\{f(x,z)+\nu \cdot(h(x)-z)\}\,.
\end{align}
\end{lemma}
\begin{proof}
We will need several properties of the function $F:C\to(-\infty,\infty]$ defined by
\begin{align}
F(z)=\sup_{\nu\in\mathbb{R}^k}\inf_{x\in X}\{f(x,z)+\nu \cdot(h(x)-z)\}\,.
\end{align}
The function $(z,x)\mapsto f(x,z)+\nu\cdot(h(x)-z)$ is convex and so $z\mapsto \inf_{x\in X}\{f(x,z)+\nu \cdot(h(x)-z)\}$ is convex for all $\nu$ (see Proposition 2.22 in \cite{rockafellar2009variational} and note that it never equals $-\infty$ by assumption (iv)).  This holds for all $\nu$ and so $F$ is convex (see Proposition 2.9 in \cite{rockafellar2009variational}).  We extend $F$ by $F|_{C^c}\equiv\infty$; this extension is also convex.

Next we show that $F$ is finite on $h(X)$: Let $z_0\in h(X)$.  Then there exists $x_0\in X$ with $h(x_0)=z_0$ and for any $\nu$ we have 
\begin{align}
\inf_{x\in X}\{f(x,z_0)+\nu \cdot(h(x)-z_0)\}\leq f(x_0,z_0)+\nu \cdot(h(x_0)-z_0)= f(x_0,z_0)\,.
\end{align}
Hence
\begin{align}
F(z_0)=\sup_{\nu\in\mathbb{R}^k}\inf_{x\in X}\{f(x,z_0)+\nu \cdot(h(x)-z_0)\}\leq  f(x_0,z_0)<\infty\,.
\end{align}

Finally we show that $F(z)=\infty$ for all  $z\in\overline{ h(X)}^c\cap C$: Let $z_0\in\overline{ h(X)}^c\cap C$.  The separating hyperplane theorem implies that there exists $v\in\mathbb{R}^k$, $v\neq 0$, and $c\in(0,\infty)$ such that $v\cdot(z_0-w)\geq c$ for all $w\in  h(X)$.  Letting $\nu=-tv$, $t>0$, we have
\begin{align}
F(z_0)\geq& \inf_{x\in X}\{f(x,z_0)-tv \cdot(h(x)-z_0)\} = \inf_{x\in X}\{f(x,z_0)+tv \cdot(z_0-h(x))\}\\
\geq &tc+\inf_{x\in X}f(x,z_0)\,.\notag
\end{align}
From assumption (iv) we see that $\inf_{x\in X}f(x,z_0)>-\infty$ and so by taking $t\to\infty$ we find $F(z_0)=\infty$.

Now we prove the claimed equality \eqref{eq:constrained_disint}.  If $h(X)=\{z_0\}$ for some $z_0$ then 
\begin{align}
\inf_{x\in {X}} f(x,h(x))=\inf_{x\in {X}} f(x,z_0)
\end{align}
 and
\begin{align}
&\inf_{z\in K}\sup_{\nu\in\mathbb{R}^k}\inf_{x\in X}\{f(x,z)+\nu \cdot(h(x)-z)\}=
\inf_{z\in K}\{\inf_{x\in X}f(x,z)+\sup_{\nu\in\mathbb{R}^k}\{\nu \cdot(z_0-z)\}\}\\
=&\inf_{z\in K}\{\inf_{x\in X}f(x,z)+\infty 1_{z\neq z_0}\}=\inf_{x\in X}f(x,z_0)\,.\notag
\end{align}
This proves the claim in this case.

Now suppose $h(X)^o\neq \emptyset$: As a first step, we have
\begin{align}
\inf_{x\in X}f(x,h(x))=\inf_{z\in h(X)}\inf_{x\in X:h(x)=z}f(x,z)\,.
\end{align}
Fix $z_0\in h(X)^o$. $h(X)$ is convex, hence for any $\tilde z\in h(X)$ and any $t\in[0,1)$ we have $(1-t)z_0+t\tilde z\in h(X)^o$ (see Theorem 6.1 in \cite{rockafellar1970convex}). Take $x_0,\tilde x\in X$ with $h(x_0)=z_0$ and  $h(\tilde x)=\tilde z$.  Then   $(1-t)x_0+t\tilde x\in X$, $h((1-t)x_0+t\tilde x)=(1-t)z_0+t\tilde z$ and
\begin{align}
&\inf_{z\in h(X)^o}\inf_{x\in X:h(x)=z}f(x,h(x))\leq f((1-t)x_0+t\tilde x,(1-t)z_0+t\tilde z)\\
\leq& tf(\tilde x, \tilde z)+(1-t)f(x_0,z_0)\notag
\end{align}  
for all $t\in[0,1)$.  Taking $t\nearrow 1$ we find
\begin{align}
& \inf_{z\in h(X)^o}\inf_{x\in X:h(x)=z}f(x,h(x))\leq f(\tilde x,\tilde z)
\end{align} 
for all $\tilde x\in X$ with $h(\tilde x)=\widetilde z$ and hence
\begin{align}
\inf_{x\in X}f(x,h(x))=&\inf_{z\in h(X)}\inf_{x\in X:h(x)=z}f(x,h(x))= \inf_{z\in h(X)^o}\inf_{x\in X:h(x)=z}f(x,z)\,.
\end{align}
For $z\in h(X)^o$ the Slater conditons hold for the convex function $f(\cdot,z)$ and affine constraint $h(\cdot)-z$ (see Theorem 8.3.1 and Problem 8.7  in \cite{luenberger1997optimization}) and so we have strong duality:
\begin{align}
\inf_{x\in X:h(x)=z}f(x,z)=\sup_{\nu\in\mathbb{R}^k}\inf_{x\in X}\{f(x,z)+\nu\cdot(h(x)-z)\}\,.
\end{align}
Therefore $\inf_{x\in X}f(x,h(x))=  \inf_{z\in h(X)^o}F(z)$. The properties of $F$ proven above imply   
\begin{align}\label{eq:dom_F}
h(X)\subset \dom F\subset \overline{h(X)}
\end{align}
 and hence $(\dom F)^o=h(X)^o$ (see Theorem 6.3 in \cite{rockafellar1970convex}).  Therefore Lemma \ref{lemma:convex_min_ri} implies $\inf_{h(X)^o}F=\inf_{(\dom F)^o}F=\inf_{\mathbb{R}^k} F$ and so for any $K$ with $h(X)\subset K\subset C$ we have
\begin{align}
\inf_{x\in X}f(x,h(x))=  \inf_{z\in K}F(z)
\end{align}
as claimed.

Finally when $h(X)$ contains more than one element and $h(X)^o=\emptyset$ we will transform the problem into an equivalent one with that fits under the previously proven case. Intuitively, in this case $h(X)$ lies in a hyperplane and has nonempty relative interior. Hence by using an affine transformation we can push it down to a lower dimensional space where it has nonempty interior. More precisely, there exists $m\in\mathbb{Z}^+$ and affine maps $\Phi:\mathbb{R}^m\to\aff(h(X))$, $\Psi:\mathbb{R}^k\to\mathbb{R}^m$ such that $\Phi^{-1}=\Psi|_{\aff(h(X))}$ and $\Psi( h(X)))^o \neq \emptyset$. 
Write $\Phi(\cdot)=A(\cdot)+a$ and $\Psi(\cdot)=B(\cdot)+b$ with $A,B$ linear and $a,b$ constant. The above properties  imply that $\Phi\circ\Psi\circ h|_X=h|_X$ and $(B h(X))^o\neq\emptyset$. Define
\begin{enumerate}
\item $\widetilde C=Bh(X)$, a nonempty convex subset of $\mathbb{R}^m$,
\item $\widetilde f:X\times\widetilde C\to\mathbb{R}$, $\widetilde f(x,z)=f(x,a+A(b+z))$, a convex map,
\item $\widetilde h:V\to\mathbb{R}^m$, $\widetilde h=Bh$, a linear map with $\widetilde h(X)= \widetilde C$.
\end{enumerate}
For any $\nu\in\mathbb{R}^m$, $\tilde z\in\widetilde C$ we have $\tilde z=Bz$ for some $z\in h(X)\subset C$, hence
\begin{align}
&\inf_{x\in X}\{\widetilde f(x,\tilde z)+\nu\cdot(\widetilde h(x)-\tilde z)\}=\inf_{x\in X}\{f(x,a+A(b+Bz))+\nu\cdot B(h(x)-z)\}\\
=&\inf_{x\in X}\{f(x,\Phi(\Psi(z)))+ (B^T\nu)\cdot (h(x)-z)\}=\inf_{x\in X}\{f(x,z)+ (B^T\nu)\cdot (h(x)-z)\}>-\infty\notag
\end{align}
by assumption (iv).   In addition, $\widetilde h(X)^o=\widetilde C^o=(Bh(X))^o\neq\emptyset$. Therefore  $\widetilde f,\widetilde h$ satisfy the assumptions (i)-(iv) of this lemma and  this system falls under the previously proven case.  Hence  
\begin{align}
\inf_{x\in X}f(x,h(x))=&\inf_{x\in X}\widetilde f(x,\widetilde h(x))=\inf_{\tilde z\in \widetilde C}\sup_{\tilde \nu\in\mathbb{R}^m}\inf_{x\in X}\{\widetilde f(x,\tilde z)+\tilde \nu \cdot(\widetilde h(x)-\tilde z)\}\\
=&\inf_{z\in h(X)}\sup_{\tilde\nu\in\mathbb{R}^m}\inf_{x\in X}\{f(x,\Phi(\Psi(z)))+ \tilde\nu\cdot ( Bh(x)-Bz)\}\notag\\
=&\inf_{z\in h(X)}\sup_{\tilde\nu\in\mathbb{R}^m}\inf_{x\in X}\{f(x,z)+ (B^T\tilde\nu) \cdot ( h(x)-z)\}\notag\\
\leq &\inf_{z\in h(X)} \sup_{\nu\in\mathbb{R}^k}\inf_{x\in X}\{f(x,z)+\nu \cdot ( h(x)-z)\}\notag\\
\leq&\inf_{z\in h(X)}\inf_{x\in X:h(x)=z}f(x,z)=\inf_{x\in X}f(x,h(x))\,,\notag
\end{align}
where the last inequality is due to weak duality.

Therefore we have proven
\begin{align}
\inf_{x\in X}f(x,h(x))= &\inf_{z\in  h(X)}\sup_{\nu\in\mathbb{R}^k}\inf_{x\in X}\{f(x,z)+\nu \cdot ( h(x)-z)\}=\inf_{z\in  h(X)}F(z).
\end{align}
From Lemma \ref{lemma:convex_min_ri} we find $\inf_{\ri(\dom F)}F=\inf_{\mathbb{R}^k} F$. From \req{eq:dom_F} we see that $\dom F\cap \ri(h(X))=\ri(h(X))\neq\emptyset$ and so $\dom F$ is not contained in the relative boundary of $h(X)$.  Therefore Corollary 6.5.2 in \cite{rockafellar1970convex} implies $\ri(\dom F)\subset \ri(h(X))$ and hence for any $K\subset \mathbb{R}^k$ with $h(X)\subset K\subset C$  we have 
\begin{align}
\inf_{x\in X}f(x,h(x))=&\inf_{z\in  h(X)}F(z)=\inf_{z\in K} F(z)\\
=&\inf_{z\in K}\sup_{\nu\in\mathbb{R}^k}\inf_{x\in X}\{f(x,z)+\nu \cdot(h(x)-z)\}\,,\notag
\end{align}
which proves the result in this final case. 
\end{proof}

\section*{Acknowledgments}
The research of J.B. was partially supported by NSF TRIPODS CISE-1934846.  J.B. would like to thank Luc Rey-Bellet for helpful discussions.

\bibliographystyle{amsplain}
\bibliography{refs}

\end{document}